\documentclass[10pt,reqno]{amsart}
\usepackage{hyperref}
\usepackage{cite}

\usepackage{comment}
\usepackage{graphicx}
\usepackage{amssymb}
\usepackage{amsfonts}
\usepackage[]{amsmath}
\usepackage[]{epsfig}
\usepackage{mathtools}
\mathtoolsset{showonlyrefs}
\usepackage[]{float}
\usepackage{setspace}
\usepackage{tikz}
\newtheorem{theorem}{Theorem}[section]

\newtheorem{definition}{Definition}[section]
\newtheorem{remark}{Remark}[section]
\newtheorem{proposition}{Proposition}[section]
\newtheorem{corollary}{Corollary}[section]

\numberwithin{equation}{section}
\newcommand{\R}{\mathbb R}

\newcommand{\eps}{\varepsilon}

\usepackage[english, activeacute]{babel}
\usepackage{amsmath,amsthm,amsxtra}
\usepackage{epsfig}
\usepackage{amssymb}
\usepackage{latexsym}
\usepackage{amsfonts}
\usepackage{hyperref}
\usepackage{pdftricks}
\newtheorem{thm}{Theorem}[section]

\newtheorem{prop}[thm]{Proposition}

\theoremstyle{remark}

\numberwithin{equation}{section}

\newcommand{\supp}{\operatorname{supp}}

\def\bm{\left( \begin{array}{cc}}
\def\endm{\end{array}\right)}

\newcommand{\be}{\begin{equation}}
\newcommand{\ee}{\end{equation}}
\newcommand{\ba}{\left(\begin{array}{c}}
\newcommand{\ea}{\end{array}\right)}
\newcommand{\bea}{\begin{eqnarray}}
\newcommand{\eea}{\end{eqnarray}}
\newcommand{\bee}{\begin{eqnarray*}}
\newcommand{\eee}{\end{eqnarray*}}
\newcommand{\ben}{\begin{enumerate}}
\newcommand{\een}{\end{enumerate}}

\newcommand{\ep}{\varepsilon}

\newcommand{\qtq}[1]{\quad\text{#1}\quad}

\begin{document}
\pagenumbering{arabic}	
\title[Wave equation with localized nonlinearity]{Scattering and nonlinear recovery for the \\ 1$d$ NLW with localized nonlinearity}
\author[L. Campos]{Luccas Campos}
\address{Department of Mathematics, Federal University of Minas Gerais}
\email{luccascampos@gmail.com}

\author[J. Murphy]{Jason Murphy}
\address{Department of Mathematics, University of Oregon}
\email{jamu@uoregon.edu}

\author[R. Scarpelli]{Renzo Scarpelli}
\address{Department of Mathematics, Federal University of Minas Gerais}
\email{renzoscb@ufmg.br}

\begin{abstract}
We consider the wave equation in 1$d$ with a spatially-localized cubic nonlinearity. We establish scattering in a weighted space and demonstrate the injectivity of the scattering map, with a corresponding stability estimate for nonlinear recovery.
\end{abstract}

\maketitle

\section{Introduction}
We consider the following one-dimensional wave equation:
\begin{equation}\label{NONLIN}
\begin{cases}
\Box u = Ku^3, \\
(u,\partial_t u)|_{t=0}= (u_0,u_1).
\end{cases}
\end{equation}
Here $u:\R_t\times \R_x\to\R$, $\Box := -\partial_t^2 + \partial_x^2$ and $K:\R\to\R$. Using standard energy arguments, one can readily obtain global well-posedness for \eqref{NONLIN} for $(u_0,u_1)\in H^1\times L^2$ and $K\in L^\infty$ (see Proposition~\ref{P:GWP} below). 

The case $K\equiv 0$ reduces to the $1d$ linear wave equation, which may be solved explicitly by the d'Alembert formula.  In particular, all solutions are sums of traveling waves of the form
\[
u(t,x) = f(x+t)+g(x-t). 
\] 
The case $K\equiv 1$ instead yields the defocusing power-type model previously considered in works such as \cite{Tao, Wei}. In this case, solutions exhibit quite different long-time behavior.  Indeed, solutions decay (in $L_x^\infty$, say) as $|t|\to\infty$.  In particular, the defocusing nonlinearity induces a nonlinear decay mechanism, which may be realized through the use of Morawetz estimates. 

Our interest in this work is to study \eqref{NONLIN} in the intermediate case, in which nonlinearity is present but restricted to a finite region of space.  To obtain the nonlinear decay mechanism, we assume that the function $K$ is nonnegative and satisfies a repulsivity condition.  The precise assumptions we impose are as follows. 

\begin{definition}\label{D:admissible}
We call $K: \R \to \R$ \emph{admissible} if $K$ is smooth, nonnegative, compactly supported in $[-1,1]$, and satisfies (i) $K(0)>0$ and (ii) the repulsivity condition $xK'(x)\leq 0$ for all $x\in\R$. 
\end{definition}

For \eqref{NONLIN} with admissible $K$, it is then natural to ask about the behavior of solutions as $|t|\to\infty$: in particular, do solutions decay, or do they scatter to linear solutions?  This seems like a strange question at first, as decay and scattering typically go hand-in-hand.  However, in light of the discussion above, this is indeed the right question to ask. 

Our first result is a scattering result for initial data $(u_0,u_1)\in H^1\times L^2$ of compact support.  In particular, we see that when the nonlinearity is localized in space, solutions do not decay globally in $L_x^\infty$ (see e.g. Remark~\ref{R:nodecay}) and instead exhibit linear behavior as $t\to\infty$, provided we measure the solution in a suitable topology. 

To state our results precisely, we define 
\begin{equation}\label{XY}
\begin{aligned}
X&=\{(u_0,u_1)\in H^1\times L^2:(u_0,u_1)\text{ are compactly supported}\},\\
Y&= \Dot{H}^1(\mathbb{R},\langle x\rangle^{-1}\,dx)\times L^2(\mathbb{R},\langle x\rangle^{-1}\,dx),
\end{aligned}
\end{equation}
and we use $S(t)$ to denote the linear solution operator (see \eqref{Group} below). 

Our first result is the following: 

\begin{theorem}\label{T1} Let $K$ be admissible and $(u_0,u_1)\in X$.  Let $u$ be the corresponding global solution to \eqref{NONLIN}.  Then there exist unique $(U_+,V_+)\in Y$ such that 
\begin{equation}\label{eq:definition_scattering_in_X}
\lim_{t \to \infty} \| S(-t)(u(t),\partial_t u(t))-(U_+,V_+)\|_Y = 0.
\end{equation}
\end{theorem}

Let us remark that while Theorem~\ref{T1} asserts that solutions exhibit asymptotically linear behavior (as measured in the $Y$-norm), the nonlinearity does have a measurable effect as $t\to\infty$.  In particular, we will show in Appendix~\ref{S:Appendix} that any solution with initial data in $X$ exhibits pointwise decay on any bounded spatial interval as $t\to\infty$.  This is in contrast to the linear equation,  for which soluti	ons fail to decay locally whenever the initial velocity has nonzero total integral.  In particular, the use of the weighted space $Y$ seems to be natural in the $1d$ setting, as it accommodates such non-decaying linear behavior.

Using Theorem~\ref{T1}, given any admissible $K$ we can define the \emph{scattering map} 
\[
\Phi:X\to Y,\qtq{given by}\Phi((u_0,u_1))=(U_+,V_+).
\]
The scattering map encodes information about the coefficient $K$ through the nonlinear interaction over time. In fact, our second main result shows that the scattering map defined via Theorem~\ref{T1} uniquely determines the coefficient $K$. 

\begin{theorem}\label{T2} Let $K_1,K_2$ be admissible, and let $\Phi_1,\Phi_2:X\to Y$ be the corresponding scattering maps. Then
\begin{equation}\label{stability}
\|K_2 - K_1\|_{L^\infty} \lesssim \biggl[\sup_{(u_0,u_1)\in X\backslash\{0\}} \frac{\|\Phi_2((u_0,u_1))-\Phi_1((u_0,u_1))\|_Y}{\|(u_0,u_1)\|_{\dot H^1\times L^2}}\biggr]^{\frac{1}{8}}, 
\end{equation}
with implicit constants depending on $K_1$ and $K_2$. In particular, if $\Phi_1\equiv\Phi_2$ then $K_1\equiv K_2$.
\end{theorem} 

We have chosen the nonlinearity to be cubic and the inhomogeneity to be supported in the interval $[-1,1]$ for the sake of simplicity.  Our arguments would apply equally well to power-type nonlinearities of the form $|u|^p u$ and with inhomogeneities compactly supported in any bounded interval.  Some decay assumption on $K$ certainly seems to be necessary (cf. the discussion of \eqref{NONLIN} with $K\equiv 1$ above), but it is likely that the compact support assumption could be relaxed.  This could be an interesting direction for follow-up work.  It would also be of interest to optimize spaces $X$ and $Y$ used to obtain scattering in Theorem~\ref{T1}.

As mentioned above, our interest in Theorem~\ref{T1} stems from the related works \cite{Tao, Wei}.  Theorem~\ref{T2} then fits in the context of recent work on the problem of the recovering unknown nonlinearities from the scattering behavior of solutions (see e.g. \cite{Gunther, ChenMurphy1, HKV, KMV1, Murphy, Sun, Watanabe, Xie, Alexakis, Sasaki}).  A particularly interesting feature of the model \eqref{NONLIN} is the complete lack of dispersion for the underlying linear model.  Nonetheless, we can establish a Morawetz estimate for \eqref{NONLIN} that yields scattering, and we can further show that strategies related to nonlinear recovery such as the Born approximation may be brought to bear in our setting. 

In the rest of the introduction, we will briefly outline the strategy of the proofs of Theorems~\ref{T1}~and~\ref{T2}. 

For Theorem~\ref{T1}, the essential ingredient is a Morawetz estimate modeled after the one appearing in \cite{Wei}.  In that work, the authors treated the $1d$ NLW with a defocusing power-type nonlinearity and proved a global space-time estimate for solutions with initial data belonging to a weighted space.  Adapting their arguments in the presence of a nonnegative repulsive inhomogeneity leads to the estimate
\begin{equation}\label{intro:Mor}
\int_0^\infty \int_\R \tfrac{K(x)}{1+t}[u(t,x)]^4\,dx\,dt \lesssim 1. 
\end{equation}
We note that to prove this estimate requires $L_{t,x}^\infty$ control over the solution.  In the setting of \cite{Wei}, one can obtain such control via conservation of energy and the Gagliardo--Nirenberg inequality.  In our setting, conservation of energy only yields $L_x^\infty$-bounds locally in space.  To extend these bounds to all of $\R$, we work with compactly supported initial data and leverage the compact support of $K$ to observe that outside of a compact interval the values of solutions are simply transported along characteristic lines for all time. For more details, see Proposition~\ref{global bound} (for the uniform estimates) and Proposition~\ref{PROP:Morawetz} (for the Morawetz estimate). 

With an estimate of the form \eqref{intro:Mor} in place, we are in a position to prove the scattering result, Theorem~\ref{T1}.  The idea is straightforward: after pulling back by the linear group $S(t)$, we need only show that the remaining integrals in the Duhamel formula converge in $Y$ as $t\to\infty$.  The Morawetz estimate provides sufficient control to establish such convergence. For the details, see Section~\ref{S:proof}. 

Let us finally discuss the proof of the Theorem~\ref{T2}.  As in many recent works on nonlinear recovery via the scattering map, the proof utilizes the Born approximation (see e.g. \cite{ChenMurphy1, HKV, KMV1, Murphy, Sun, Xie}).  This entails first using the Duhamel formula to obtain an implicit formula for the scattering map and then approximating the full nonlinear solution by its first Picard iterate (i.e. the linear solution with the same initial conditions). This leads to the following expansion for the second component $V_+$ of the scattering map:
\begin{equation}\label{Vexpand}
\begin{aligned}
V_+(x) - u_1(x) &= - \tfrac12\int_0^\infty K(x+s)u_L^3(s,x+s) \,ds \\
& \quad -\tfrac12\int_0^\infty K(x-s)u_L^3(s,x-s) \,ds \\
& \quad + \text{higher order terms}, 
\end{aligned}
\end{equation}
where $(u_0,u_1)$ are the initial conditions and $u_L$ is the linear solution with the same data.  In the formula above, the `higher order terms' are those resulting from replacing $u^3$ with $u_L^3$, which we expect to be of size
\[
|u^3 - u_L^3| = \mathcal{O}(|u|^2|u-u_L|) = \mathcal{O}(|u|^5).
\]
In particular, we expect the Born approximation to be most effective in the small-data regime.  In our setting, we make use of small right-traveling linear waves of the form
\[
u_{L,n}(t,x) = n^{-1}\varphi(n(x-x_0-t))
\]
for $\varphi\in C_c^\infty$. We then take the inner product of \eqref{Vexpand} with $n^4\varphi(n(x-x_0))$.  We will show that the quantity arising from the first term on the right-hand side may be used to recover the antiderivative $\int_{x_0}^\infty K(y)\,dy$, while the remaining terms may be treated perturbatively. We note this argument requires an improved $L_x^\infty$ estimate for short times, which we establish in Proposition~\ref{P:ST-small}.  Continuing from above, we derive that if two inhomogeneities $K_1$ and $K_2$ produce the same scattering map, they must agree identically.  In fact, as the argument can be made quantitative, we can obtain the stability estimate asserted in Theorem~\ref{T2}.  We remark that this approach of obtaining a stability estimate from bounds on antiderivatives of the inhomogeneities is reminiscent of the strategy introduced recently in the work \cite{Xie}.  For the details, see Section~\ref{S:proof}. 

The rest of this paper is organized as follows: Section~\ref{S:prelim} contains some preliminary material. Section~\ref{S:GWP} contains a proof of global well-posedness for \eqref{NONLIN}.  Section~\ref{S:bounds} contains the proof of global $L_{t,x}^\infty$ estimates for solutions to \eqref{NONLIN}.  Section~\ref{S:Morawetz} then contains the proof of the Morawetz estimate.  Section~\ref{S:proof} finally contains the proofs of the main results, Theorems~\ref{T1}~and~\ref{T2}.  In Appendix~\ref{S:Appendix}, we provide a proof of local $L_x^\infty$ decay (see Theorem~\ref{THM:pwdecay}). 

\subsection*{Acknowledgements} Much of this work was carried out while L.C. and R.S. were visiting scholars 
at the University of Oregon, with financial support from CNPq (Conselho Nacional de Desenvolvimento Cient\'{i}fico e Tecnol\'{o}gico). L. C. was partially supported by the CNPq grants 07733/2023-8 and 404800/2024-6, and by the FAPEMIG (Funda\c{c}\~{a}o de Amparo \`{a} Pesquisa do Estado de Minas Gerais) grant APQ-03186-24. J. M. was supported by Simons Foundation grant MPS-TSM-00006622. R.S. was supported by the CNPq grant 201049/2025-2 and by FAPEMIG.

\section{Preliminaries}\label{S:prelim}

We write $A\lesssim B$ to indicate $A\leq CB$ for some $C>0$.  We use the standard notations for Lebesgue and Sobolev spaces, and we use the standard Japanese bracket notation: $\langle x\rangle=\sqrt{1+x^2}$. 
% \subsection{Linear wave equation 1D}

The solution to the $1d$ linear wave equation
\begin{align}\label{LIN}
\begin{cases}
\Box u = 0 \\
(u, \partial_t u){|_{t=0}} =(u_0,u_1) \in \Dot{H}^1(\R) \times L^2(\R).  
\end{cases}
\end{align}
is given by the well-known d'Alembert formula
\begin{equation}\label{lineard}
u_L(t,x) = \tfrac{1}{2}(u_0(x+t) + u_0(x-t)) + \tfrac{1}{2}\int_{x-t}^{x+t} u_1(s)\,ds.
\end{equation}
We may equivalently write
\[
(u_L(t),\partial_t u_L(t))=S(t)(u_0,u_1),
\]
where $S(t)$ is the unitary operator on $\dot H^1\times L^2$ given by
\begin{equation}\label{Group}
S(t) = \left(\begin{array}{cc} \cos(t|\nabla|) & |\nabla|^{-1}\sin(t|\nabla|)\\ -|\nabla|\sin(t|\nabla|) & \cos(t|\nabla|) \end{array}\right).
\end{equation}

The solution to \eqref{NONLIN} is formally given by the Duhamel formula
\[
u(t) = \cos(t |\nabla|)u_0 + |\nabla|^{-1}\sin(t |\nabla|) u_1 + \int_0^t |\nabla|^{-1}\sin((t-s)|\nabla|)K[u(s)]^3\,ds,
\]
which can also be written as the d'Alembert formula 
\begin{equation}\label{Duhamel}
u(t,x) = u_L(t,x)- \tfrac{1}{2}\int_0^t \int_{x-(t-s)}^{x+(t-s)} K(y) u^3(s,y)\,dy\,ds,
\end{equation}
where $u_L$ is as in \eqref{lineard}.

We define the energy density, momentum density, and momentum current, respectively, as:
\begin{align}
\begin{cases}
T_{00} = \frac{1}{2}(\partial_x u)^2 + \frac{1}{2}(\partial_t u)^2 + \frac{1}{4}Ku^4 \\
T_{01} = \partial_x u\,\partial_t u \\
T_{11} = \frac{1}{2}(\partial_x u)^2 + \frac{1}{2}(\partial_t u)^2 -\frac{1}{4}Ku^4.
\end{cases}
\end{align}
Solutions to the equation \eqref{NONLIN} then enjoy the following local conservation laws:
\begin{align}\label{ConsLaws}
\begin{cases}
\partial_t T_{00} = \partial_xT_{01} \\
\partial_tT_{01}= \partial_xT_{11}+  \frac{1}{4}K'u^4,
\end{cases}
\end{align}
which in particular implies the conservation of the energy defined by
\[
E(u,\partial_t u) = \tfrac{1}{2} \int |\partial_x u(t,x)|^2\, dx + \tfrac{1}{2} \int |\partial_t u(t,x)|^2\, dx + \tfrac{1}{4} \int K(x)[u(t,x)]^{4}\,dx.
\]

Moreover, defining 
\begin{equation}\label{Lpm}
L_+ = \partial_t + \partial_x \quad and \quad L_- = \partial_t - \partial_x,
\end{equation}
we can readily deduce the following relations for solutions to \eqref{NONLIN}:
\begin{equation}
\begin{cases}
T_{00} -T_{11} = \frac12 Ku^4\\
T_{00}-T_{01}= \frac{1}{2}(L_-u)^2 +\frac{1}{4}Ku^4\\
T_{00}+T_{01}= \frac{1}{2}(L_+u)^2 +\frac{1}{4}Ku^4.
\end{cases}
\end{equation}

We will also make use of the conservation of energy adapted to the bounded region between $t=0$ and $t=T$ on the outgoing null line $x=t+R$ (see \cite{Wei}):
\begin{align}\label{EQ:adaptedconslaw}
\int_0^T \tfrac{1}{2}(L_+u)^2 + \tfrac{1}{4}K u^4 \big|_{x=t+R} \,dt = \int_{t+R}^\infty T_{00}(t,x)\,dx  \big|_{t=T}^{t=0} \lesssim E(u(t))
\end{align}
Similarly, we have the analogous bound for the ingoing null line $x= -t -R$:
\begin{align}\label{EQ:adaptedconslaw2}
\int_0^T \tfrac{1}{2}(L_-u)^2 + \tfrac{1}{4}K u^4 \big|_{-x=t+R} \,dt = \int^{-t-R}_{-\infty} T_{00}(t,x)\,dx  \big|_{t=T}^{t=0} \lesssim E(u(t)).
\end{align}

%%%%
\section{Global well-posedness}\label{S:GWP}

Global well-posedness in $H^1\times L^2$ for \eqref{NONLIN} with $K\in L^\infty$ follows from standard energy arguments.  We include a proof for the sake of completeness.

To simplify the formulas below, we introduce the notation
\begin{equation}\label{UW}
U(t) = \cos(t |\nabla|)\qtq{and} W(t) =|\nabla|^{-1}\sin(t |\nabla|).
\end{equation}

\begin{proposition}[Global well-posedness]\label{P:GWP} Let $K\in L^\infty$.  Then \eqref{NONLIN} is globally well-posed for initial data $(u_0,u_1)\in H^1\times L^2$. 
\end{proposition}

\begin{proof} Let $M=\|(u_0,u_1)\|_{H^1\times L^2}$ and let $T=T(K,M)>0$ be a small parameter to be determined below. We define the complete metric space
\[
B = \{u\in C([0,T];H^1)\cap C^1([0,T];L^2): \|u\|_{L_t^\infty H_x^1} + \|\partial_t u\|_{L_t^\infty L_x^2} \leq 4CM\}
\]
equipped with the distance
\[
d(u,v) = \|u-v\|_{L_t^\infty H_x^1}+\|\partial_t u - \partial_t v\|_{L_t^\infty L_x^2}.
\]
Here $C>0$ is a universal constant and all norms are taken over $[0,T]\times\R$. 

We now define the Duhamel operator
\[
\Phi(u) = U(t)u_0 + W(t) u_1 + \int_0^t W(t-s)K\,[u(s)]^3\,ds. 
\]
We wish to prove that for $T$ chosen sufficiently small, $\Phi$ is a contraction on $B$.

To begin, we let $u\in B$. Estimating the norms of operators $U(t)$ and $W(t)$ via the Plancherel Theorem, we first obtain
\begin{align*}
\|\Phi(u(t))\|_{H^1} & \leq \|u_0\|_{H^1} + (1+T)\|u_1\|_{L^2} + T(1+T)\|K u^3\|_{L_t^\infty L_x^2} \\
& \leq M + (1+T)M + T(1+T)\|K\|_{L^\infty}\|u\|_{L_t^\infty L_x^2} \|u\|_{L_{t,x}^\infty}^2 \\
& \leq M + (1+T)M + CT(1+T)\|K\|_{L^\infty} M^3
\end{align*}
uniformly in $t\in[0,T]$.  Thus for $T=T(K,M)$ sufficiently small, we have
\[
\|\Phi u\|_{L_t^\infty H_x^1} \leq 2CM.
\]
Similarly,
\begin{align*}
\|\partial_t \Phi(u(t))\|_{L^2} &\leq \|u_0\|_{\dot H^1} + \|u_1\|_{L^2} + T\|Ku^3\|_{L_t^\infty L_x^2} \\
&\leq 2M + CT\|K\|_{L^{\infty}} M^3 
\end{align*}
uniformly in $t\in[0,T]$. 

Choosing $T=T(K,M)$ sufficiently small, we obtain
\[
\|\Phi(u)\|_{L_t^\infty H_x^1} + \|\partial_t \Phi(u)\|_{L_t^\infty L_x^2} \leq 4CM,
\]
so that $\Phi:B\to B$. 

Given $u,v\in B$, similar estimates yield
\begin{align*}
\|\Phi(u)-\Phi(v)\|_{L_t^\infty H_x^1} & \lesssim T(1+T) \|K\|_{L^\infty}\{\|u\|_{L_{t,x}^\infty}^2+\|v\|_{L_{t,x}^\infty}^2\}\|u-v\|_{L_t^\infty L_x^2} \\
& \lesssim T(1+T)\|K\|_{L^\infty} M^2 \|u-v\|_{L_t^\infty L_x^2}
\end{align*}
and
\begin{align*}
\|\partial_t[\Phi(u)-\Phi(v)]\|_{L_t^\infty L_x^2} \lesssim T\|K\|_{L^\infty} M^2\|u-v\|_{L_t^\infty L_x^2}.
\end{align*}
Thus for $T=T(K,M)$ sufficiently small, we obtain
\[
d(\Phi(u),\Phi(v))\leq\tfrac12 d(u,v).
\]  

As $\Phi:B\to B$ is a contraction, we conclude that there exists unique $u\in B$ satisfying $u=\Phi(u)$, yielding a local solution to \eqref{NONLIN} on $[0,T]$.  We remark that similar estimates can also be used to obtain continuity of the data-to-solution map. 

To pass from local to global well-posedness, we will obtain an \emph{a priori} estimate on the growth of the $H^1\times L^2$ norm.  In fact, the $\dot H^1\times L^2$-norm of any solution $(u,u_t)$ remains bounded due to conservation of energy. For the $L^2$-norm of a solution $u$ with initial data $(u_0,u_1)$, we compute
\[
\tfrac{d}{dt} \|u(t)\|_{L^2}^2 = 2\int u(t,x)\partial_t u(t,x)\,dx \leq 2\|u_t\|_{L_t^\infty L_x^2} \|u(t)\|_{L_x^2},
\]
which yields the \emph{a priori} estimate
\[
\|u(t)\|_{L_x^2} \leq \|u_0\|_{L^2} + Ct\|u_1\|_{L^2}.
\]
We can therefore iterate the local-in-time result to obtain global existence. \end{proof}

\section{Uniform estimates}\label{S:bounds}

%{\texorpdfstring{$L^\infty$}{L\textasciicircum∞} bounds}

The local theory developed in the previous section ensures that the solution $u$ to \eqref{NONLIN} belongs to $C_t H_x^1$, but does not provide uniform control on the $L_x^2$-norm.  In particular, we have $u\in C_t L_x^\infty$ by Sobolev embedding, but the estimates appearing in the local theory only provide the estimate $\|u(t)\|_{L_x^\infty} \lesssim \langle t\rangle^{\frac12}$.

In this section, we firstly prove that we can establish uniform $L_{t,x}^\infty$ bounds for solutions with compactly supported initial data.  We will also establish an improved short-time estimate for small data. 

We note that in \cite{Tao}, the authors use the conservation of energy combined with the Gagliardo--Nirenberg inequality to control the $L_x^\infty$-norm for solutions with general data in $H^1\times L^2$.  This approach fails in our setting due to the fact that the nonlinearity is localized to the unit interval. Our approach will be to first establish $L_{t,x}^\infty$-bounds for $x\in[-R,R]$ and then to use the support assumptions on the initial data to extend these bounds to all $x\in\R$. 

\begin{proposition}\label{global bound} Let $R>0$ and suppose $(u_0,u_1)\in H^1\times L^2$ satisfies $\supp(u_0,u_1)\subset[-R,R]$ and $E((u_0,u_1))=E_0$. Let $K$ be admissible in the sense of Definition~\ref{D:admissible}, and let $u$ denote the global solution to \eqref{NONLIN} with initial data $(u_0,u_1)$.  Then
\begin{equation}\label{Linftybound}
\|u\|_{L_{t,x}^\infty([0,\infty)\times\R)} \lesssim_{K,E_0,R} 1. 
\end{equation}
\end{proposition}

\begin{proof}[Proof of Proposition~\ref{global bound}] Let $u,K$ be as in the statement of the theorem. We will first prove an estimate on $[0,\infty)\times[-R,R]$.

To this end, we first define
\[
 I_K = \{x\in\R: K(x) \geq \tfrac12K(0) \},
\]
which satisfies $|I_K|\sim_K 1$ and (without loss of generality) $I_K\subset[-R,R]$. Let $\chi_K$ denote a smooth cutoff to $I_K$.  Then, by Sobolev embedding and conservation of energy,
\begin{align*}
\|u(t) \chi_K \|_{L^\infty} &\lesssim \|u(t)\chi_K\|_{L^2}^{\frac12}[\|u(t)\partial_x[\chi_K]\|_{L^2}+\|\partial_x u(t) \chi_K\|_{L^2}]^{\frac12} \\
& \lesssim |I_K|^{-\frac12}\|u(t)\|_{L^2(I_K)} + \|u(t)\|_{L^2(I_K)}^{\frac12}\|\partial_x u\|_{L_t^\infty L_x^2}^{\frac12} \\
& \lesssim_K \|u(t)\|_{L^2(I_K)}+E_0^{\frac14} \|u(t)\|_{L^2(I_K)}^{\frac12}.
\end{align*}
To estimate the $L^2$-norm on $I_K$, we use the definition of $I_K$, Cauchy--Schwarz, and the conservation of energy:
\begin{align*}
\int_{I_K} |u(t,x)|^2\,dx &\lesssim_K \int_{I_K} [K(x)]^{\frac12} |u(t,x)|^2\,dx \\
& \lesssim_K |I_K|^{\frac12}\biggl[\int K(x)|u(t,x)|^4\,dx\biggr]^{\frac12}  \lesssim_{K,E_0} 1. 
\end{align*}
Thus we obtain
\[
\|u\|_{L_{t,x}^\infty([0,\infty)\times I_K)} \lesssim_{K,E_0} 1.
\]

We now use the Fundamental Theorem of Calculus, Cauchy--Schwarz, and the fact that $0\in I_K$ to estimate
\begin{align*}
|u(t,x)| &\leq |u(t,0)|+\biggl|\int_0^{x_0}\partial_y u(t,y)\,dy\biggr| \\
& \leq \|u\|_{L_{t,x}^\infty([0,\infty)\times I_K)} + |x|^{\frac12} \|\partial_x u\|_{L_t^\infty L_x^2} \\
& \lesssim_{K,E_0,R} 1
\end{align*}
uniformly in $t\geq 0$ and $x\in[-R,R]$. Thus we obtain 
\begin{equation}\label{local bound2} 
\| u\|_{L_{t,x}^\infty([0,\infty)\times [-R,R])} \lesssim_{K,E_0,R}1.
\end{equation}

We now consider the external region. We first demonstrate how to obtain bounds for 
\[
(t,x)\in\Omega_R:=[0,\infty)\times[R,\infty).
\]
To begin, we recall the notation in \eqref{Lpm} and define
\[
w_+ = L_+ u,\qtq{which satisfies} L_- w_+ = -Ku^3. 
\]
In particular, for any $(t,x)$ we have
\[
\tfrac{d}{ds} w_+ (s, x+t-s) = -K(x+t-s)[u(s,x+t-s)]^3,
\]
so that
\begin{equation}\label{est1}
w_+(t,x) = w_+(0,x+t) - \int_0^t K(x+t-s)[u(s,x+t-s)]^3\,ds.
\end{equation}
Now, if $x\geq R>1$, then (since $K$ is supported in $[-1,1]$) the integral in \eqref{est1} vanishes identically. We therefore obtain 
\begin{equation}\label{MOC}
\partial_t u(t,x) + \partial_x u(t,x) = u_1(x+t)+u_0'(x+t)=0,\quad (t,x)\in \Omega_R. 
\end{equation}

Now if $(t,x)\in\Omega_R$ satisfies $x\geq t+R$, then by \eqref{MOC} and the method of characteristics, we obtain
\[
u(t,x)=u_0(0,x-t) =0.
\]
If instead $(t,x)\in\Omega_R$ satisfies $R\leq x\leq t+R$, then \eqref{MOC} and the method of characteristics imply
\[
u(t,x) = u(t-(x-R),R).
\]
Thus we conclude
\[
\|u\|_{L_{t,x}^\infty(\Omega_R)} \leq \|u\|_{L_{t,x}^\infty([0,\infty)\times[-R,R])}. 
\]

An analogous argument using $w_-:=L_- u$ yields bounds for $(t,x)\in[0,\infty)\times(-\infty,-R]$. Thus we see that \eqref{local bound2} implies \eqref{Linftybound}, which completes the proof of Proposition~\ref{global bound}. \end{proof}

\begin{remark}\label{R:nodecay} The argument above establishes several important properties.  First, we obtain finite speed of propagation: if $(u_0,u_1)$ are supported in $[-R,R]$, then the corresponding solution $u$ vanishes for $|x|>R+|t|$.  Second, we see that solutions do not decay globally in space.  Indeed, nonzero values of the solutions are transported along characteristic lines for all time.  On the other hand, in Appendix~\ref{S:Appendix} we will establish $L_x^\infty$ decay on any bounded spatial interval. \end{remark}

We turn to the second result of this section. 

\begin{proposition}\label{P:ST-small} Fix $T_0>0$ and let $K$ be admissible in the sense of Definition~\ref{D:admissible}. There exists $\eps_0>0$ sufficiently small  that the following holds:  If $(u_0,u_1)\in H^1\times L^2$ are such that the solution $u_L$ to \eqref{LIN} with $(u_L,\partial_t u_L)|_{t=0}=(u_0,u_1)$ satisfies
\[
\|u_L\|_{L_{t,x}^\infty([0,\infty)\times\R)} < \eps<\eps_0,
\]
then the solution $u$ to \eqref{NONLIN} with $(u,\partial_t u)|_{t=0}=(u_0,u_1)$ satisfies
\[
\|u-u_L\|_{L_{t,x}^\infty([0,T_0]\times\R)} \lesssim_K T_0\eps^3.
\]
\end{proposition}

\begin{proof} Let $T_0,K,u_L,u$ be as in the statement of the theorem, with $\eps_0>0$ to be determined below. Define 
\[
J=\{T\in[0,T_0]:\|u\|_{L_{t,x}^\infty([0,T]\times\R)} \leq 2\eps\}.
\]

We observe that $0\in J$ and that $J$ is closed by continuity of the flow in $L_x^\infty$.  Now, by using the d'Alembert formula \eqref{Duhamel}, we have the following estimate for any $T\in(0,T_0]$:
\[
\|u\|_{L_{t,x}^\infty([0,T]\times\R)} \leq \eps + T\|K\|_{L^1} \|u\|_{L_{t,x}^\infty([0,T]\times\R)}^3. 
\] 
Thus if $T\in J$ and $\eps_0$ is sufficiently small, we have that
\[
\|u\|_{L_{t,x}^\infty([0,T]\times\R)} \leq \eps + [8T_0\|K\|_{L^1}\eps_0^2] \eps\leq \tfrac32\eps.
\]
By continuity of the flow, we find that $J$ contains an open interval around $T$. By connectedness of $[0,T_0]$, we conclude that $J=[0,T_0]$, and using \eqref{Duhamel} as above now yields
\[
\|u-u_L\|_{L_{t,x}^\infty([0,T_0]\times\R)} \leq T_0\|K\|_{L^1} \|u\|_{L_{t,x}^\infty}^3 \lesssim_K T_0 \eps^3,
\] 
as desired.\end{proof}

%%%%%%%%%%%%%%
\section{Morawetz estimate}\label{S:Morawetz}

In this section we will prove a Morawetz estimate inspired by \cite{Wei}.  As a consequence, we will derive Corollary~\ref{rem:mor}, which will be a key ingredient in the proof of Theorem~\ref{T1} in the following section.

 In what follows, we denote 
\[
a = t+1-x\qtq{and} b= t+1+x,
\]
and we set 
\[
Q = -L_+L_- (u^2) = -L_-L_+(u^2)=\square(u^2).
\]

\begin{prop}\label{PROP:Morawetz} Let $(u_0,u_1)\in H^1\times L^2$ with $\supp(u_0,u_1)\subset[-R,R]$ and $E((u_0,u_1))=E_0$.  Let $K$ be admissible in the sense of Definition~\ref{D:admissible}, and let $u$ denote the global solution to \eqref{NONLIN} with initial data $(u_0,u_1)$. Then 
\begin{equation}\label{eq:localized_Morawetz}
\int_0^{\infty} \int_{-t-1}^{t+1} \frac{(t+1)^2-x^2}{(t+1)^3}K(x)[u(t,x)]^4\,dx\,dt \lesssim_{K,E_0,R} 1.
\end{equation}
\end{prop}

\begin{proof} The proof follows \cite{Wei} fairly closely.

First, direct calculation shows that 
\begin{align}
\begin{cases}
L_+(L_-u)^2 = -\frac{1}{2}K(x)L_-(u^4)\\
L_-(L_+u)^2 = -\frac{1}{2}K(x)L_+(u^4),
\end{cases}
\end{align}
which imply
\begin{equation}\label{Morawetz1}
\begin{cases}
L_+(\frac{a^2}{(t+1)^2}(L_-u)^2) + 2 \frac{K(x)}{4}L_-(\frac{a^2}{(t+1)^2}u^4) = \frac{-2a^2}{(t+1)^3}(L_-u)^2 + \frac{abK(x)}{(t+1)^3} u^4, \\
L_-(\frac{b^2}{(t+1)^2}(L_+u)^2) + 2 \frac{K(x)}{4}L_+(\frac{b^2}{(t+1)^2}u^4) = \frac{-2b^2}{(t+1)^3}(L_+u)^2 + \frac{abK(x)}{(t+1)^3} u^4.
\end{cases}
\end{equation} 
Defining
\begin{align*}
P_- &= \frac{a^2}{(t+1)^2}(L_-u)^2 + \frac{2b^2}{(t+1)^2}\frac{1}{4}K(x)u^4,\\
P_+ &= \frac{b^2}{(t+1)^2}(L_+u)^2 + \frac{2a^2}{(t+1)^2}\frac{1}{4}K(x)u^4,
\end{align*}
we sum the identities in \eqref{Morawetz1} and use $Q=\Box(u^2)$ and the repulsivity condition $xK'(x)\leq 0$ to obtain
\begin{align}\label{4.1}
\frac{2ab}{(t+1)^3} K(x)u^4 
% &\leq \frac{2abQ}{(t+1)^3} - \left[L_+(P_-)-\frac{b^2}{2(t+1)^2}u^4K'(x) \right] \\&\quad-\left[ L_-(P_+) + \frac{a^2}{2(t+1)^2}u^4K'(x)\right] \nonumber \\
&\leq  \frac{2abQ}{(t+1)^3} - L_+(P_-) - L_-(P_+) + \frac{2xK'(x)u^4}{(t+1)} \nonumber \\
&\leq \frac{2abQ}{(t+1)^3} - L_+(P_-) - L_-(P_+).
\end{align}

Now, we integrate the expression \eqref{4.1} in the domain
\[
\Sigma_1^T = \{ 0\leq t \leq T, \quad|x| \leq t+1 \}.
\]
The left-hand side corresponds to the left-hand side of \eqref{eq:localized_Morawetz}.  Thus it remains to estimate the resulting terms on the right-hand side. 

We first consider the term in \eqref{4.1} involving $Q$.  Using Stokes' theorem, we begin by writing
\[
\int_{\Sigma_1^T} \tfrac{ab}{(t+1)^3} Q \,dx\,dt = \int_{\Sigma_1^T} u^2\, \Box \bigl(\tfrac{ab}{(t+1)^3}\bigr) \,dx \,dt + \int_{\partial \Sigma_1^T} w_1,
\]
where
\[
\omega_1 = \bigl[\tfrac{ab\,\partial_x (u^2)}{(t+1)^{3}}  - u^2 \partial_x\bigl(\tfrac{ab}{(t+1)^{3}}\bigr)\bigr]\,dt + \bigl[\tfrac{ab\,\partial_t(u^2)}{(t+1)^{3}}  -  u^2 \partial_t\bigl(\tfrac{ab}{ (t+1)^{3}}\bigr)\bigr]\,dx.
\]
As $|x| \leq (t+1)$, direct computation gives the bound
\[
\bigl|\Box \bigl(\tfrac{ab}{(t+1)^3}\bigr) \bigr|\leq \tfrac{8}{(t+1)^3}.
\]
Thus
\[
\biggl|\int_{\Sigma_1^T} u^2 \,\Box (\tfrac{ab}{(t+1)^3}) \, dx\, dt \biggr|\lesssim \|u\|_{L_{t,x}^\infty}^2 \int_0^T\int_{-t-1}^{t+1} \tfrac{1}{(t+1)^3} \,dx\,dt \lesssim \|u\|_{L_{t,x}^\infty}^2.
\]

To compute the integral of $w_1$ on the boundary, we split it into four components $\{\Gamma_j\}_{j=1}^4$. 
In the null segments
\[
\Gamma_1 = \{x=t+1, \quad 0\leq t\leq T\} \qtq{and} \Gamma_2 = \{x=-1-t, \quad 0\leq t\leq T\}
\]
the form $w_1$ is zero. On the constant time slices 
\[
\Gamma_3 = \{t=0, \quad |x|\leq 1\} \qtq{and} \Gamma_4 = \{t= T, \quad |x|\leq T+1 \},
\]
the differential form reduces to
\[
w_1 = (ab(t+1)^{-3} \partial_t(u^2) -  u^2 \partial_t(ab (t+1)^{-3}))\,dx.
\]
Noting that for $|x|\leq (t+1)$,
\[
|ab (t+1)^{-3}| \leq (t+1)^{-1}, \quad |\partial_t(ab (t+1)^{-3})| \leq 2 (t+1)^{-2},
\]
we use conservation of energy to obtain
\begin{align}
\int_{-t-1}^{t+1} & | (ab(t+1)^{-3} \partial_t(u^2) -  u^2 \partial_t(ab (t+1)^{-3})) | \,dx\\ &\lesssim \int_{-t-1}^{t+1} (t+1)^{-1}| \partial_t u||u| + (t+1)^{-2} |u|^2 dx  \\
& \lesssim \|u\|_{L_{t,x}^\infty} E_0^{\frac12} + \|u\|_{L_{t,x}^\infty}^2.
\end{align}

We turn to the contribution of the $P_\pm$ terms in \eqref{4.1}.  By Stokes' Theorem,
\begin{align}
-\int_{\Sigma_1^T}&  L_+P_- + L_-P_+\,dx\,dt = \int_{\partial\Sigma_1^T} (P_-+P_+)\,dx - (P_--P_+)\,dt.\\
&= \int_{-t-1}^{t+1} (P_- + P_+)\,dx \bigr|_{t=T}^{t=0} + 2 \int_{0}^T P_+(t,t+1) + P_-(t,-t-1)\, dt.
\end{align}

We now use the conservation of energy adapted to the null lines $\Gamma_1$ and $\Gamma_2$ defined above (cf. \eqref{EQ:adaptedconslaw} and \eqref{EQ:adaptedconslaw2}). Noting that 
\[
P_+(t,t+1) = 4(L_+u)^2 \leq 16T_{00}(t,t+1),
\]
we have 
\[
\int_0^T T_{00}(t,t+1) \,dt = \int_{t+1}^{\infty} T_{00}(t,x)\,dx \big|_{t=T}^{t=0} \leq 2E_0.
\]
Similarly, 
\[
\int_0^T P_-(t,-t-1)\,dt \lesssim E_0.
\]
Finally, noting that in the region $|x|\leq t+1$ we have
\[
P_- + P_+ \leq 4((L_+ u)^2 + (L_-u)^2) + 8 \cdot\tfrac{1}{4} K(x)u^4,
\]
we obtain
\[
\left| \int_{-t-1}^{t+1} (P_- + P_+)\,dx \big|_{t=T}^{t=0} \right| \leq 16 E_0.
\]
Combining the estimates above and appealing to Proposition~\ref{global bound}, we complete the proof of \eqref{eq:localized_Morawetz}. \end{proof}

Using \eqref{eq:localized_Morawetz} and the fact that $\text{supp}(K)\subset[-1,1]$, we can obtain the following estimate, which will be the key input for the proof of scattering in the following section.
\begin{corollary}\label{rem:mor} Let $u,K$ be as in the statement of Proposition~\ref{PROP:Morawetz}. Then
\[
\int_0^\infty \int_{-1}^1 \tfrac{K(x)}{1+t}[u(t,x)]^4 \,dx\,dt \lesssim 1.
\]
\end{corollary}

%%%%%%%%%%%%%%%%%%
\section{Proofs of the main results}\label{S:proof}
In this section we prove our main results, Theorems~\ref{T1}~and~\ref{T2}.  We recall the spaces $X,Y$ introduced in \eqref{XY}.

\begin{proof}[Proof of Theorem~\ref{T1}] Let $u$ be the solution to \eqref{NONLIN} as in the statement of Theorem~\ref{T1}. Recalling \eqref{Group}, we can write 
\begin{align}
S(-t)(u(t),\partial_t u(t))
&= (u_0,u_1) - (N(t), M(t)).
\end{align}
where
\begin{equation}\label{NM}
\begin{aligned}
N(t) &:= \tfrac12 \int_0^t \int_{x-s}^{x+s} K(y)[u(s,y)]^3\,dy\,ds, \\
M(t) &:=\tfrac12 \int_0^t K(x+s)[u(s,x+s)]^3 + K(x-s)[u(s,x-s)]^3\,ds. 
\end{aligned}
\end{equation}
Theorem~\ref{T1} is then equivalent to the assertion that $(N(t),M(t))$ has a limit in $Y$ as $t\to\infty$. Thus we define
\begin{align*}
\tilde N(t) &:= \int_t^\infty \int_{x-s}^{x+s} K(y)[u(s,y)]^3\,dy\,ds, \\
\tilde M(t) &:= \int_t^\infty K(x+s)[u(s,x+s)]^3+K(x-s)[u(s,x-s)]^3\,ds 
\end{align*}
and aim to show that $(\tilde N(t),\tilde M(t))\to 0$ in $Y$ as $t\to\infty$. 

We first compute
\[
\partial_x \tilde N(t) = \int_t^{\infty}K(x+s)[u(s,x+s)]^3\,ds - \int_t^{\infty} K(x-s)[u(s,x-s)]^3\,ds
\]
and focus on estimating the first integral, as the computations used to treat the second are similar.  We estimate by duality and fix $\varphi \in L^2(\R;\langle x\rangle\,dx)$. Using the support properties of $K$, Fubini's Theorem, and changing variables (first via $y=x+s$ and second via $\tau=y-x$), we have
\begin{align*}
I:= \langle \partial_x \tilde N(t),\varphi\rangle&= \int_\R \int_t^\infty K(x+s)[u(s,x+s)]^3\varphi(x)\,ds\,dx \\
& =\int_\R \int_{t+x}^\infty K(y)[u(y-x,y)]^3\varphi(x)\,dy\,dx \\
& = \int_{-1}^1 \int_{-\infty}^{y-t} K(y)[u(y-x,y)]^3\varphi(x)\,dx\,dy \\
& = \int_{-1}^1 \int_t^\infty K(y)[u(\tau,y)]^3\varphi(y-\tau)\,d\tau\,dy.  
\end{align*}
We now apply Cauchy--Schwarz and Proposition~\ref{global bound} to obtain
\begin{align*}
|I|&  \lesssim \biggl(\int_{-1}^1 \int_t^\infty\tfrac{K^2(y)}{1+\tau}[u(\tau,y)]^6\,d\tau\,dy\biggr)^{\frac12}\biggl(\int_{-1}^1 \int_t^\infty (1+\tau)[\varphi(y-\tau)]^2\,d\tau\,dy\biggr)^{\frac12}. \\
&\lesssim \|K\|_{L_x^\infty}^{\frac12}\|u\|_{L_{t,x}^\infty}\biggl(\int_{-1}^1 \int_t^\infty\tfrac{K(y)}{1+\tau}[u(\tau,y)]^4\,d\tau\,dy\biggr)^{\frac12}\\
&\quad\quad\quad\quad\quad\quad\quad\quad\quad\times \biggl(\int_{-1}^1 \int_t^\infty (1+\tau)[\varphi(y-\tau)]^2\,d\tau\,dy\biggr)^{\frac12}.
\end{align*}
%\begin{align*}
%|I|& \lesssim \int_{-1}^1 \biggl(\int_t^\infty \tfrac{K^2(y)}{1+\tau}[u(\tau,y)]^6\,d\tau\biggr)^{\frac12}\biggl(\int_t^\infty (1+\tau)[\varphi(y-\tau)]^2\,d\tau\biggr)^{\frac12}\,dy  \\
%&\lesssim \|K\|_{L_x^\infty}^{\frac12}\|u\|_{L_{t,x}^\infty}\biggl(\int_{-1}^1 \int_t^\infty\tfrac{K(y)}{1+\tau}[u(\tau,y)]^4\,d\tau\,dy\biggr)^{\frac12}\\
%&\quad\quad\quad\quad\quad\quad\quad\quad\quad\times \biggl(\int_{-1}^1 \int_t^\infty (1+\tau)[\varphi(y-\tau)]^2\,d\tau\,dy\biggr)^{\frac12} \\
%& \lesssim \biggl(\int_{-1}^1 \int_t^\infty\tfrac{K(y)}{1+\tau}[u(\tau,y)]^4\,d\tau\,dy\biggr)^{\frac12}\biggl(\int_{-1}^1 \int_t^\infty (1+\tau)[\varphi(y-\tau)]^2\,d\tau\,dy\biggr)^{\frac12}.
%\end{align*}
By Corollary~\ref{rem:mor}, we have
\[
\int_{-1}^1 \int_t^\infty\tfrac{K(y)}{1+\tau}[u(\tau,y)]^4\,d\tau\,dy = o(1) \qtq{as}t\to\infty. 
\]
On the other hand, by a change of variables and the fact that $y\in[-1,1]$,
\begin{align*}
\biggl|\int_{-1}^1 \int_t^\infty (1+\tau)[\varphi(y-\tau)]^2\,d\tau\,dy \biggr|& = \biggl|\int_{-1}^1\int_{-\infty}^{y-t} (1+y-x)[\varphi(x)]^2\,dx\,dy\biggr| \\
& \lesssim \int_\R \langle x\rangle[\varphi(x)]^2\,dx. 
\end{align*}
Using Proposition~\ref{global bound} as well, it follows that
\[
|I|\lesssim \|\varphi\|_{L^2(\langle x\rangle\,dx)} \cdot o(1) \qtq{as}t\to\infty. 
\]

We therefore derive that
\[
\|\partial_x \tilde N(t)\|_{L^2(\langle x\rangle^{-1}\,dx)} = o(1) \qtq{as}t\to\infty.
\]
As similar arguments show that
\[
\|\tilde M(t)\|_{L^2(\langle x\rangle^{-1}\,dx)} = o(1) \qtq{as}t\to\infty,
\]
we obtain the convergence $(\tilde N(t),\tilde M(t))\to 0$ in $Y$ as $t\to\infty$.  In particular, defining $(N_+,M_+)$ as the complete time integrals in \eqref{NM}, we obtain that $(N(t),M(t))\to (N_+,M_+)$ in $Y$ as $t\to\infty$.  

We therefore obtain the desired scattering statement with
\begin{equation}\label{defplus}
(U_+,V_+) := (u_0-N_+,u_1-M_+),
\end{equation}
which completes the proof of Theorem~\ref{T1}.\end{proof}

As described in the introduction, Theorem~\ref{T1} allows us to define the scattering map $\Phi:X\to Y$ given by $\Phi((u_0,u_1))=(U_+,V_+)$.  Our second main result, Theorem~\ref{T2}, asserts that the scattering map uniquely determines the nonlinear coefficient in \eqref{NONLIN}.

\begin{proof}[Proof of Theorem~\ref{T2}] Let $K_1, K_2$ be admissible inhomogeneities, with corresponding scattering maps $\Phi_1,\Phi_2:X\to Y$.  It suffices to establish the stability estimate \eqref{stability}.

Let us fix $x_0\in[-2,2]$ and choose $\varphi\in C_c^\infty([-1,1];\R)$ such that
\[
\int_\R [\varphi(x)]^4\,dx = 1. 
\]
For each $n\geq 1$, we consider the initial condition $(v_{0,n},v_{1,n})\in X$ defined by
\[
v_{0,n}(x) = n^{-1} \varphi(n(x-x_0)),\quad v_{1,n}(x) = -\varphi'(n(x-x_0)), 
\]
which leads to the following right-traveling solution to \eqref{LIN}:
\begin{equation}\label{uLdef}
u_{L,n}(t,x) = n^{-1}\varphi(n(x-x_0-t)),
\end{equation}
which we observe satisfies
\begin{equation}\label{uLn-bds}
\|u_{L,n}\|_{L_{t,x}^\infty([0,\infty)\times\R)} \lesssim n^{-1}. 
\end{equation}

Now let us define the solutions $u_{1,n},u_{2,n}$ to
\[
\begin{cases} \Box u_{j,n} = K_j u_{j,n}^3, \\ (u_{j,n},\partial_t u_{j,n})|_{t=0}=(v_{0,n},v_{1,n})\end{cases}
\]
and define the scattered states $(U_{j,n},V_{j,n})\in Y$ by
\[
(U_{j,n},V_{j,n}):=\Phi_j((v_{0,n},v_{1,n})).
\]

Now, for all $n$ sufficiently large we may apply Proposition~\ref{P:ST-small} to obtain 
\begin{equation}\label{ujn-bds}
\|u_{j,n}\|_{L_{t,x}^\infty([0,4]\times\R)} \lesssim n^{-1} \qtq{and} \|u_{j,n} - u_{L,n}\|_{L_{t,x}^\infty([0,4]\times\R)} \lesssim n^{-3}
\end{equation}
for $j\in\{1,2\}$, uniformly in $n$. 

Using \eqref{NM}--\eqref{defplus}, we have that for $j\in\{1,2\}$, 
\begin{align*}
V_{j,n}(x) & = v_{1,n}(x)-\tfrac12\int_0^\infty K_j(x+s)u_{j,n}^3(s,x+s) + K_j(x-s)u_{j,n}^3(s,x-s)\,ds \\
& = v_{1,n}(x) - \tfrac12\int_0^\infty K_j(x+s)u_{L,n}^3(s,x+s)+K_j(x-s)u_{L,n}^3(s,x-s)\,ds \\
& \quad + \tfrac12 G_{j,n}(x), 
\end{align*}
where
\[
G_{j,n}(x):=\int_0^\infty K_j(x+s)(u_{L,n}^3 - u_{j,n}^3)(s,x+s)+ K_j(x-s)(u_{L,n}^3 -u_{j,n}^3)(s,x-s)\,ds.
\]
Thus, denoting
\[
V_n = 2[V_{1,n}-V_{2,n}],\quad K=K_1-K_2, \qtq{and} G_n = G_{1,n}-G_{2,n},
\]
we may write 
\begin{align*}
 \int_0^\infty & K(x+s)u_{L,n}^3(s,x+s)\,ds \\ 
&\quad = - V_n(x) - \int_0^\infty K(x-s)u_{L,n}^3(s,x-s)\,ds  + G_n(x).
\end{align*}

We now take an inner product of this expression with $n^{4}\varphi(n(x-x_0))$. Recalling \eqref{uLdef}, this yields 
\begin{align}
 \int_0^\infty&\int_\R K(x+s) n [\varphi(n(x-x_0))]^4\,dx\,ds \label{pft2-main} \\
& = -n^4 \int_\R \varphi(n(x-x_0))V_n(x) \,dx \label{pft2-v}\\
& \quad  - \int_0^\infty\int_\R K(x-s) n\varphi(n(x-x_0))[\varphi(n(x-2s-x_0))]^3\,dx\,ds \label{pft2-1}\\
& \quad + n^{4}\int_\R \varphi(n(x-x_0))G_n(x)\,dx. \label{pft2-2}
\end{align}
We will show that \eqref{pft2-main} may be used to recover $\int_{x_0}^\infty K$, while \eqref{pft2-v} may be controlled by the difference of the scattering maps and \eqref{pft2-1}--\eqref{pft2-2} may be treated perturbatively. 

We begin by estimating the error terms.  We claim that for $n$ sufficiently large, 
\begin{equation}\label{T2error1}
|\eqref{pft2-1}| \lesssim n^{-1} \qtq{and}|\eqref{pft2-2}| \lesssim n^{-2}.
\end{equation}

For the first estimate in \eqref{T2error1}, we begin by changing variables $y=n(x-x_0)$ and $\tau = 2ns$ to write
\begin{align*}
\eqref{pft2-1} = -\tfrac12 n^{-1} \int_0^\infty\int_\R K(x_0+\tfrac{2y-\tau}{2n})\varphi(y)[\varphi(y-\tau)]^3\,dy\,d\tau.
\end{align*}
Recalling the support properties of $\varphi$, we see that for the integrand to be nonzero we must have $y$ and $\tau$ restricted to bounded intervals, which yields the desired estimate.

For the second estimate in \eqref{T2error1}, we first observe that using the support properties of $\varphi$ and the fact that $x_0\in[-2,2]$, we have
\[
\biggl| n^4 \int_\R G_n(x) \varphi(n(x-x_0))\,dx \biggr| \lesssim_\varphi n^3\|G_n\|_{L_x^\infty(|x|\leq 3)}
\]
for all $n$ large. Recalling the definition of $G_n$, we see that it will suffice to estimate the quantities
\[
\int_0^\infty K_j(x+s)(u_{L,n}^3-u_{j,n}^3)(s,x+s) + K_j(x-s)(u_{L,n}^3 - u_{j,n}^3)(s,x-s)\,ds
\]
in $L_x^\infty(|x|\leq 3)$.  Recalling that $\text{supp}(K_j)\subset[-1,1]$, we see that for $|x|\leq 3$ this integral may be restricted to $s\in[0,4]$. Using \eqref{uLn-bds} and \eqref{ujn-bds}, we can therefore obtain that 
\begin{align*}
n^3\|G_n\|_{L_x^\infty(|x|\leq 3)} & \lesssim n^3 \sum_{j=1}^2 \|u_{L,n}^3-u_{j,n}^3\|_{L_{t,x}^\infty([0,4]\times\R)} \\
& \lesssim n^3 \sum_{j=1}^2(\|u_{L,n}\|_{L_{t,x}^\infty}^2 + \|u_{j,n}\|_{L_{t,x}^\infty([0,4]\times\R)}^2)\|u_{j,n}-u_{L,n}\|_{L_{t,x}^\infty([0,4]\times\R)} \\
& \lesssim n^{-2},
\end{align*}
as desired. 

We turn to the estimate of \eqref{pft2-v}. To simplify the notation, let us denote
\[
\|\Phi_2-\Phi_1\|_{X\to Y} = \sup_{(u_0,u_1)\in X\backslash\{0\}} \frac{\|\Phi_2((u_0,u_1))-\Phi_1((u_0,u_1))\|_{Y}}{\|(u_0,u_1)\|_{\dot H^1\times L^2}}. 
\]
Using the support properties of $\varphi$, we then have 
\begin{equation}\label{T2errorv}
\begin{aligned}
|\eqref{pft2-v}| & \lesssim n^4 \|\varphi(n(\cdot-x_0))\|_{L^2(\langle x\rangle\,dx)} \|V_n\|_{L^2(\langle x\rangle^{-1}\,dx)} \\
& \lesssim n^4 \|\varphi(n\cdot)\|_{L^2} \|\Phi_2-\Phi_1\|_{X\to Y} \|(v_{0,n},v_{1,n})\|_{\dot H^1\times L^2} \\
& \lesssim n^3 \|\Phi_2-\Phi_1\|_{X\to Y}. 
\end{aligned}
\end{equation}

Finally, we rewrite \eqref{pft2-main} as
\begin{align*}
\eqref{pft2-main} = \int_\R \biggl[\int_{x_0}^\infty K(s+\tfrac{y}{n})\,ds\biggr] \, [\varphi(y)]^4\,dy.  
\end{align*}
Using the support properties of $K$ and $\varphi$, we observe that the integral over $s\in[x_0,\infty)$ may be restricted to $[-4,4]$ for all $n$ sufficiently large. Thus (recalling the normalization of $\varphi$ and using the Fundamental Theorem of Calculus) we may write
\begin{align*}
\biggl|\eqref{pft2-main} - \int_{x_0}^\infty K(s)\,ds \biggr| & \leq 8 \int_\R  |\tfrac{y}{n}| \|K'\|_{L^\infty} |\varphi(y)|^4\,dy \lesssim n^{-1}.  
\end{align*}
Thus, returning to \eqref{pft2-main}--\eqref{pft2-2}, applying \eqref{T2error1} and \eqref{T2errorv}, we obtain
\[
\biggl|\int_{x_0}^\infty [K_1(s)-K_2(s)]\,ds\biggr| \lesssim n^3\|\Phi_2-\Phi_1\|_{X\to Y} + n^{-1}. 
\]
Choosing $n$ to be a large multiple of $\|\Phi_2-\Phi_1\|_{X\to Y}^{-\frac14}$, this yields
\[
\biggl|\int_{x_0}^\infty [K_1(s)-K_2(s)]\,ds\biggr|  \lesssim \|\Phi_2-\Phi_1\|_{X\to Y}^{\frac14}. 
\]
As $x_0\in[-2,2]$ was arbitrary, we obtain that 
\[
\| \mathcal{K}_2 - \mathcal{K}_1\|_{L_x^\infty([-2,2])} \lesssim \|\Phi_2-\Phi_1\|_{X\to Y}^{\frac14}, \qtq{where} \mathcal{K}_j(x):=\int_x^\infty K_j(s)\,ds.
\]

To pass to an estimate on $K_2-K_1$, we recall the support properties of $K$ and use the Landau--Kolmogorov inequality on the interval $[-2,2]$.  This yields
\begin{align*}
\|K_2-K_1\|_{L^\infty} &\lesssim \|\mathcal{K}_2 - \mathcal{K}_1\|_{L^\infty([-2,2])}^{\frac12}\|K_2'-K_1'\|_{L^\infty}^{\frac12} \\
& \lesssim_{K_1,K_2} \|\Phi_2-\Phi_1\|_{X\to Y}^{\frac18},
\end{align*}
which implies \eqref{stability}, as desired. \end{proof}

\appendix

%%%%%%%%
\section{$L^\infty$ decay on bounded intervals}\label{S:Appendix}

In Section~\ref{S:bounds} we established uniform $L_x^\infty$ bounds for compactly supported initial data in $H^1\times L^2$.  In this section we show that the $L_x^\infty$-norm on any bounded spatial interval decays as $t\to\infty$. 

\begin{theorem}\label{THM:pwdecay} Let $(u_0,u_1)\in H^1\times L^2$ with $\supp(u_0,u_1)\subset[-R_0,R_0]$.  Let $K$ be admissible in the sense of Definition~\ref{D:admissible}, and let $u$ denote the global solution to \eqref{NONLIN} with initial data $(u_0,u_1)$.  Then for any $R\geq 1$, 
\[
\lim_{t \to \infty} \|u(t)\|_{L_x^\infty([-R,R])} = 0.
\]
\end{theorem}

\begin{proof} We first claim that 
\[
\lim_{T \to \infty} \tfrac{1}{T} \int_0^T \int_{-1}^1 K(x)[u(t,x)]^4 \,dx\,dt = 0.
\]
To see this, we recall Corollary~\ref{rem:mor}, which implies that given $\ep>0$, there exists $T_0 \geq 1$ such that for any $T>T_0$, 
\[
\tfrac{1}{T} \int_{T_0}^T \int_{-1}^1 K(x)[u(t,x)]^4\,dx\,dt \lesssim \ep.
\]
Thus, by conservation of energy,
\[
\tfrac{1}{T}\int_0^T \int_{-1}^1 K(x)[u(t,x)]^4\,dx\,dt \lesssim \tfrac{T_0}{T} E(u_0,u_1) + \ep,
\]
which implies the claim.

We now upgrade this average decay to a pointwise-in-time estimate.  

Let $R\geq 1$.  Arguing as in the proof of Proposition~\ref{PROP:Morawetz}, we define 
\[
\Sigma_R^T = \{|x|\leq t+ R, \quad 0\leq t \leq T\}
\]
and use the conservation laws \eqref{ConsLaws} to obtain 
\[
\partial_t((t+R)T_{00}+xT_{01}) - \partial_x((t+R)T_{01}+xT_{11}) = 2\tfrac{K(x)}{4}u^4 + \tfrac{xK'(x)u^4}{4}.
\]
Integrating this relation on $\Sigma_R^T$, and using Stokes' Theorem gives us
\begin{align}\label{EQ:aux1}
\int_{\partial \Sigma_R^T} [((t+R)&T_{00}+xT_{01})\,dx + ((t+R)T_{01}+xT_{11})\,dt] \\
&= -2\int_{\Sigma_R^T}\tfrac{K(x)u^4}{4}\,dx\,dt
-\int_{\Sigma_R^T} \tfrac{xK'(x)u^4}{4} \,dx\,dt.
\end{align}

The boundary can be decomposed in four parts, denoted $\{\Gamma_j\}_{j=1}^4$.  Denoting the integrand of the boundary integral by $S$, we have
\begin{align*}
S = (t+R)(L_+u)^2\,dt&\qtq{on}\Gamma_1 = \{x=t+R, \quad 0\leq t \leq T\}, \\
S = -(t+R)(L_-u)^2\,dt&\qtq{on} \Gamma_2 = \{-x=t+R, \quad 0\leq t \leq T\} , \\
S = [(t+R)T_{00}+xT_{01}]\,dx&\qtq{on}\Gamma_3 =\{t=0, \quad |x|\leq R\}, \\
S = [(t+R)T_{00}+xT_{01}]\,dx&\qtq{on}\Gamma_4 = \{t=T, \quad |x|\leq T+ R\}.
\end{align*} 
Thus
\begin{equation}\label{EQ:aux2}
\begin{aligned}
&\underbrace{\int_{-t-R}^{t+R}  [(t+R)T_{00}+xT_{01}]\,dx \big|_{t=0}^{t=T}}_{I} \\
&\quad\quad= \underbrace{-\int_0^T(t+R)[(L_+u(t,t+R))^2{+}(L_-u(t,-t-R))^2]\,dt}_{II} \\
&\quad\quad\quad+ 2\int_{\Sigma_R^T} \tfrac{K(x)u^4}{4}\,dx\,dt  + \int_{\Sigma_R^T} \tfrac{xK'(x)}{4}u^4 \,dx\,dt.
\end{aligned}
\end{equation}

Using the conservation of energy adapted to lines, we can control $II$ as follows:
\begin{align}
|II| \leq 2(T+R) \biggl[ \int_{-\infty}^{-t-R}T_{00}(t,x)\,dx \big|_{t=T}^{t=0} + \int_{t+R}^{\infty}T_{00}(t,x)\,dx \big|_{t=T}^{t=0}\biggr].
\end{align}

In the region $|x| \leq t+R$ we observe the following relations that will allow us to bound $I$ from below:
\begin{align}
\begin{cases}
(t+R)T_{00} + xT_{01} \geq (t+R)(T_{00}-|T_{01}|) \geq (t+R)\frac{K(x)u^4}{4},\\
(t+R)T_{00}+xT_{01} \leq 2(t+R)T_{00}.
\end{cases}
\end{align}

Thus, using the repulsivity condition $xK'(x)\leq 0$, \eqref{EQ:aux2} implies 
\begin{align}
\tfrac14(T+R)\int_{|x|\leq T+R} &{K(x)[u(T,x)]^4}\,dx - 2R\int_{|x|\leq R}T_{00}(0,x)\,dx \\
&\leq 2(T+R)\int_{|x|\geq t+R} T_{00}(t,x)\,dx \big|_{t=T}^{t=0}  +\tfrac12 \int_{\Sigma_R^T} {K(x)u^4} \,dx\,dt,
\end{align}
which yields 
\begin{align}
(T+R)\int_{-1}^1\tfrac14 {K(x)[u(T,x)]^4}dx &\leq 2R \int_{\mathbb{R}}  T_{00}(0,x)\,dx + 2T \int_{|x|\geq R} T_{00}(0,x)\,dx \\
&\quad + \tfrac12 \int_{\Sigma_R^T} {K(x)u^4}\,dx\,dt.
\end{align}
Dividing both sides by $T+R$, taking the limit as $T\to\infty$, and using the average potential energy decay, we obtain
\begin{align}
\limsup_{T\to \infty} \tfrac14 \int_{-1}^1 {K(x)[u(T,x)]^4}\,dx \leq 2\int_{|x|\geq R} T_{00}(0,x)\,dx.
\end{align}
Sending $R \to \infty$, we conclude 
\begin{align}\label{EQ:PED}
\lim_{T\to \infty} \int_{-1}^1 K(x)[u(T,x)]^4\,dx = 0.
\end{align}

It follows that the solution $u$ decays on the set $I_K=\{x:K(x)\geq \tfrac12K(0)\}$. Indeed, the Gagliardo--Nirenberg inequality yields
\begin{align}
\|K^{\frac14}u(t)\|_{L^\infty_x} \leq \| \partial_x [K^{\frac14}u(t)]\|_{L^2_x}^{\frac13} \|K^{\frac14}u(t)\|_{L^4_x}^{\frac23},
\end{align}
so that by conservation of energy, Proposition~\ref{global bound}, and \eqref{EQ:PED} we obtain 
\begin{align}\label{EQ:DecayWK}
\lim_{t \to \infty } \| K^{\frac14}u(t)\|_{L^\infty_x} = 0.
\end{align}

We now extend this decay to an arbitrary bounded region of space. To this end, we fix $R>1$ and let $x \in [-R,R]$. By computations as in \eqref{est1} above, we obtain that for $t$ sufficiently large, 
\begin{align}
\begin{cases}
\partial_t u(t,x) + \partial_x u(t,x) = - \displaystyle\int_0^t K(x+t-s)[u(s,x+t-s)]^3\,ds, \\
\partial_t u(t,x) - \partial_x u(t,x) = - \displaystyle\int_0^t K(x-t+s)[u(s,x-t+s)]^3\,ds.
\end{cases}
\end{align}
Changing variables, this yields 
\begin{equation}\label{char2}
\begin{cases}
\partial_t u(t,x) + \partial_x u(t,x) =  \displaystyle\int_{x}^{x+t} K(y)[u(x+t-y,y)]^3 \,dy \\
\partial_t u(t,x) - \partial_x u(t,x) = - \displaystyle\int_{x-t}^x K(y)[u(y+t-x,y)]^3 \,dy.
\end{cases}
\end{equation}
As $\supp K \subset[-1,1]$, we may therefore apply H\"older's inequality, Proposition~\ref{global bound}, and \eqref{EQ:PED} to find that for $|t|\geq 100R$,
\begin{align*}
\biggl| \int_{x}^{x+t} K&(y)u^3(x+t-y,y) \,dy\biggr|+\biggl|\int_{x-t}^{x} K(y)u^3(x+t-y,y) \,dy\biggr| \\
& \lesssim \int_{-1}^1 K(y) |u(x+t-y,y)|^3\,dy \\
& \lesssim \biggl(\int_{-1}^1 K(y) |u(x+t-y,y)|\,dy\biggr)^{\frac13} \biggl(\int_{-1}^1 K(y) |u(x+t-y,y)|^4\,dy \biggr)^{\frac23} \\
& \lesssim \|K\|_{L_x^\infty}^{\frac13}\|u\|_{L_{t,x}^\infty}^{\frac13} \sup_{T\geq \frac12t} \biggl(\int_{-1}^1 K(y)[u(T,y)]^4\,dy\biggr)^{\frac23} \\
& = o(1) \qtq{as}t\to\infty
\end{align*}
uniformly for $x\in[-R,R]$. In particular, combining the equations in \eqref{char2}, we find that
\[
\|\partial_x u(t,x)\|_{L_x^\infty([-R,R])} = o(1) \qtq{as} t\to\infty.
\]
Thus for any $x\in[-R,R]$ we may use the Fundamental Theorem of Calculus and \eqref{EQ:DecayWK} to obtain
\begin{align*}
|u(t,x)| & \leq |u(t,0)| + R\|\partial_x u(t,x)\|_{L_x^\infty([-R,R])}  \\
& \lesssim (1+R)o(1) \qtq{as}t\to\infty,
\end{align*}
which completes the proof. 
\end{proof}

%\bibliography{Ref}
%\nocite{*}
\end{document}